\documentclass[twoside]{amsart}
\usepackage[english]{babel}
\usepackage{graphicx}
\usepackage{mathtools}		
\usepackage{mathabx}		
\usepackage{indentfirst}	
\usepackage{amsmath}
\usepackage{amssymb} 
\usepackage{amsthm}
\usepackage{thmtools}
\usepackage{enumitem}		
\usepackage[backref=page,colorlinks=true, linkcolor=red, citecolor=blue, urlcolor=blue]{hyperref}
\usepackage[capitalise]{cleveref} 
\usepackage{tikz}
\usepackage[font=footnotesize]{caption}
\usepackage{dsfont}

\newcommand{\Lip}{\mathrm{Lip}}

\newcommand{\Z}{\mathbb{Z}}
\newcommand{\N}{\mathbb{N}}

\newcommand{\R}{\mathbb{R}}
\newcommand{\T}{\mathbb{T}}

\newcommand{\unitsphere}{\mathbb{S}}
\renewcommand{\epsilon}{\varepsilon}

\newcommand{\bkav}[1]{\frac{1}{n}#1^{(n)}}

\newcommand{\dd}{\,\mathrm{d}}   
\newcommand{\innt}{\,\mathrm{int}}
\newcommand{\relinnt}{\,\mathrm{relint}}
\newcommand{\per}{\,\mathrm{per}}
\newcommand{\norm}[1]{\left\lVert#1\right\rVert}
\newcommand{\conv}{\,\mathrm{conv}}
\newcommand{\Orbit}{\mathcal{O}}

\DeclareMathOperator{\Ima}{Im}

\usepackage{color}
\definecolor{darkblue}{rgb}{0.0, 0.0, 0.55}\usepackage{fancyhdr}
\usepackage[left=4.5cm,top=3.5cm,right=4.5cm,bottom=2.5cm]{geometry}

\renewcommand*{\backref}[1]{}
\renewcommand*{\backrefalt}[4]{\quad \tiny 
  \ifcase #1 (\textbf{NOT CITED.})%
  \or    (Cited on page~#2.)%
  \else   (Cited on pages~#2.)%
  \fi}

\makeatletter

\def\MRbibitem{\@ifnextchar[\my@lbibitem\my@bibitem}

\def\mybiblabel#1#2{\@biblabel{{\hyperref{http://www.ams.org/mathscinet-getitem?mr=#1}{}{}{#2}}}}

\def\myhyperanchor#1{\Hy@raisedlink{\hyper@anchorstart{cite.#1}\hyper@anchorend}}

\def\my@lbibitem[#1]#2#3#4\par{%
  \item[\mybiblabel{#2}{#1}\myhyperanchor{#3}\hfill]#4%
  \@ifundefined{ifbackrefparscan}{}{\BR@backref{#3}}%
  \if@filesw{\let\protect\noexpand\immediate
    \write\@auxout{\string\bibcite{#3}{#1}}}\fi\ignorespaces%
}

\def\my@bibitem#1#2#3\par{%
  \refstepcounter\@listctr
  \item[\mybiblabel{#1}{\the\value\@listctr}\myhyperanchor{#2}\hfill]#3%
  \@ifundefined{ifbackrefparscan}{}{\BR@backref{#2}}%
  \if@filesw\immediate\write\@auxout
    {\string\bibcite{#2}{\the\value\@listctr}}\fi\ignorespaces%
}

\makeatother

\newtheorem{thm}{Theorem}[section]
\newtheorem{prop}[thm]{Proposition}
\newtheorem{lemma}[thm]{Lemma}
\newtheorem{coro}[thm]{Corollary}
\theoremstyle{remark}
\theoremstyle{definition}

\newtheorem{qstn}[thm]{Question}
\usepackage{setspace}

\begin{document}

\title{\textbf{Generic Rotation Sets}}

\author{\small{Sebasti\'an Pavez-Molina}}

\markboth{S Pavez-Molina}{Generecity of Strictly Convex Rotation Sets}

\date{}

\maketitle

\begin{abstract}
Let $(X,T)$ be a topological dynamical system. Given a continuous vector-valued function $F \in C(X, \R^{d})$ called a \textit{potential}, we define its rotation set $R(F)$ as the set of integrals of $F$ with respect to all $T$-invariant probability measures, which is a convex body of $\R^{d}$. In this paper, we study the geometry of rotation sets. We prove that if $T$ is a non-uniquely ergodic topological dynamical system with a dense set of periodic measures, then the map $R(\cdot)$ is open with respect to the uniform topologies. As a consequence, we obtain that the rotation set of a generic potential is strictly convex and has $C^{1}$ boundary. Furthermore, we prove that the map $R(\cdot)$ is surjective, extending a result of Kucherenko and Wolf.
\end{abstract}


\section{Introduction}
Let $(X,T)$ be a topological dynamical system, that is, a compact metric space $X$ together with a continuous map $T: X \to X$. We denote by $\mathcal{M}_{T}$ the set of all $T$-invariant probability measures, which is convex and weak-$\star$ compact. Given a continuous potential $F: X \to \R^{d}$, we define its \textit{rotation set} as:
$$
R(F) = \left \{ \int F \dd \mu : \mu \in \mathcal{M}_{T} \right \} .
$$
This is a convex body in $\R^{d}$, that is, a non-empty compact and convex subset of $\R^{d}$.

This definition originates from the rotation theory on the torus \cite{MK}: if $f: \T^{d} \to \T^{d}$ is continuous, homotopic to the identity with lift $\widetilde{f}: \R^{d} \to \R^{d}$, we consider the displacement function $F(x):= \widetilde{f}(x)-x$. The corresponding rotation set $R(F)$ yields important information about the dynamics of $f$. Note that in the one-dimensional case, $R(F)= \{ \rho(\widetilde{f}) \}$, where $\rho(\cdot)$ is the Poincar\'e rotation number. For $d \geq 2$, it is known that generically the rotation set is given by a
rational polygon \cite{P}, and there are rotation sets with smooth boundary
points \cite{BCH}.

Returning to the general context, Ziemian \cite{Zi} studied the situation where the dynamics is a subshift of finite type (SFT) and the potential $F$ is locally constant, and proved that in this case the rotation set is a polytope. On the other hand, Kucherenko and Wolf \cite{KW} proved that if $T$ is a SFT then every convex body of $\R^{d}$ appears as a rotation set of a  continuous potential.

Ergodic optimization \cite{Je1,Je2} is another motivation for the study of the rotation set. Given a function  $f \in C(X)$, one is interested in the quantity
\begin{equation}\label{beta}
\beta(f)=\sup_{\mu \in \mathcal{M}_{T}} \int f \dd \mu \hspace{0.2cm} ,
\end{equation}
called the \textit{maximum ergodic average}. Any measure $\mu \in \mathcal{M}_{T}$ satisfying $\int f \dd \mu = \beta(f)$ is called an $f$-\textit{maximizing measure}. The main problem of ergodic optimization is to identify maximizing measures and to understand their properties. For generic functions in the space $C(X)$, the maximizing measure is unique; furthermore, the same holds for other spaces of functions: see \cite[Theorem 3.2]{Je1}. Note that in this case the (one-dimensional) rotation set is $R(f)=[\alpha(f), \beta(f)]$, where $\alpha(f)=-\beta(-f)$ is the \textit{minimum ergodic average}.

Consider the more general problem of computing the maximum ergodic average $\beta(f)$ for all functions $f$ in a given finite-dimensional subspace of $C(X)$,  say with generators $f_{1}, ... ,f_{d}$. If $f= \sum_{j=1}^{d} \alpha_{j} f_{j}$ then we have:
$$
\beta(f)=\sup_{\vec{x} \in R(F)} (\alpha_{1}, ..., \alpha_{d}) \cdot \vec{x}
$$
where $F=(f_{1}, f_{2}, ... ,f_{d})$. Therefore, the problem reduces to the study of the rotation set of $F$, which is called Vectorial Ergodic Optimization \cite[section 2]{B}.

Let us describe one of the first examples of rotation sets, introduced by Jenkinson \cite{Je4}. Let $X=\R / \Z$ be the circle, $T$ be the doubling map, and $F(x)=(\cos(2\pi x), \sin(2\pi x))$ be  the potential. The corresponding rotation set $R(F)$ is called \textit{the fish}. Validating experimental results of Jenkinson, Bousch \cite{Bo1} proved that the fish is strictly convex and every point on its boundary is the integral of $F$ with respect to a unique $T$-invariant probability measure. Furthermore, he proved that the fish has a dense subset of corners (points of non-differentiability), and each corner is the integral of $F$ with respect to a unique $T$-invariant porbability measure, which is \textit{periodic}, that is, supported on a single periodic orbit.

It is natural to ask whether these characteristics of the fish are typical among rotation sets: see \cite[section 2]{B} for further discussion. In this work, we give a partial answer to this question. Under a mild hypothesis on the dynamics $T$ (which is satisfied for the doubling map and SFT), we prove that for generic continuous potentials, the rotation set is strictly convex and (unlike the fish) has a $C^{1}$ boundary. This genericity result is obtained as a corollary of our main theorem, which reads as follows:

\begin{thm}\label{thmC}
Let $T: X \to X$ be a non-uniquely ergodic topological dynamical system with dense set of periodic measures. Then the map $$R: (C(X, \R^{d}), \norm{\cdot}_{\infty}) \to (\mathit{CB}(\R^{d}),d_{H})$$ that associates to each potential $F$ its rotation set $R(F)$ is continuous, open, and surjective.
\end{thm}
Here, $C(X, \R^{d})$ is endowed with the uniform norm, and $\mathit{CB}(\R^{d})$ is the set of convex bodies of $\R^{d}$ endowed with the Hausdorff distance  (see section \ref{sec2} for more details). Continuity of the map $R$ is trivial. Surjectivity of $R$ was already known when $T$ is a SFT: see \cite[Theorem 2]{KW}.

The hypothesis of denseness of periodic measures holds for any dynamical system with the specification property (e.g., uniformly expanding transformations, SFT, and Anosov diffeomorphisms). It also holds for many classes of non-hyperbolic dynamics, for example, $\beta$ shifts, $S$-gap shifts, and isolated non-trivial transitive sets of $C^{1}$-generic diffeomorphisms: see \cite{GK}.

As a consequence of our main result, we have:
\begin{coro} \label{thmA}
Let $T: X \to X$ be a non-uniquely ergodic topological dynamical system with dense set of periodic measures. Then there exists a residual subset $\mathcal{R}$ of $C(X,\R^{d})$ such that $R(F)$ is strictly convex and has $C^{1}$ boundary for all $F \in \mathcal{R}$.
\end{coro}
\begin{proof}[Proof of Corollary \ref{thmA}]
The set of convex bodies which are strictly convex with $C^{1}$ boundary is residual \cite[p. 133]{S}. Therefore the pre-image under $R$ of this set is also residual, since $R$ is continuous and open by Theorem \ref{thmC}.
\end{proof}
The $C^{1}$ regularity in the corollary cannot be improved in this case: for generic convex bodies the boundary is not $C^{1+\alpha}$, for any $\alpha>0$ : see \cite{KlN}.

It is natural to ask whether Theorem \ref{thmC} holds for spaces of more regular functions, for example, Lipschitz functions. The answer is negative: see section \ref{sec6.2}.

\section{Preliminaries} \label{sec2}
We say that a non-empty subset $K \subset \R^{d}$ is a \emph{convex body} if it is compact and convex. We denote the set of convex bodies by $\mathit{CB}(\R^{d})$, and by $\mathit{CB}_{\circ}(\R^{d})$ the set of convex bodies with non-empty interior. Aditionally, given a convex body $K$,  we denote by $\innt(K)$ its interior and $\relinnt(K)$ its relative interior. We endow $\mathit{CB}(\R^{d})$ with a structure of metric space, given by the \textit{Hausdorff distance} defined by:
$$
d_{H}(K,L)=\max  \left \{ \sup_{x \in K}  \inf_{y \in L} \norm{x-y}, \sup_{y \in L} \inf_{x \in K} \norm{x-y} \right \}.
$$
This definition only requires $K,L$ to be compact. Also, the $\sup$ and $\inf$ can be replaced by $\max$ and $\min$ due to compactness. This metric turns $\mathit{CB}(\R^{d})$ into a complete, locally compact metric space \cite[p. 62]{S}. The Hausdorff distance between two convex bodies can also be obtained just considering their boundaries, namely, if $K,L$ are two convex bodies, then 
$d_{H}(K,L)=d_{H}(\partial  K, \partial L)$ \cite[p. 61]{S}. A useful lemma that will be used later is the following:
\begin{lemma} \label{preadjustlemma}
Let $K \in \mathit{CB}(\R^{d})$ and $0<\delta<1$. Then there exists $K_{\delta} \in \mathit{CB}(\R^{d})$ such that $K_{\delta} \subset \relinnt(K)$ and $d_{H}(K_{\delta}, K) < \delta$.
\end{lemma}
\begin{proof}
Define the \textit{support function} of an arbitrary $L \in \mathit{CB}(\R^{d})$ by 
\begin{equation*}
\displaystyle h_{L}(u)= \sup_{x \in L} x\cdot u ,
\end{equation*}
and denote $\overline{h}_{L}=h_{L} |_{\unitsphere^{n-1}}$. Applying a translation if necessary, suppose that $0 \in \relinnt(K)$. Define $K_{\delta}= \left (1 - \frac{\delta}{kd} \right )K$, where $d= \sup_{x \in \partial K} \norm{x}$ and $k \in \N$ is such that $\frac{\delta}{kd}<1$. It is clear that $K_{\delta} \subset \relinnt(K)$, and using \cite[Lemma 1.8.14]{S} :
\begin{equation*}
\begin{split}
d_{H}(K_{\delta}, K) &=\norm{\overline{h}_{K_{\delta}}-\overline{h}_{K}}_{\infty} \\
&\leq \frac{\delta}{kd}\norm{\overline{h}_{K}}_{\infty} \\
&= \frac{\delta}{kd} \sup_{x \in \partial K } \norm{x} < \delta. \qedhere
\end{split}
\end{equation*} 
\end{proof}

Let $X$ be a compact metric space and let $T: X \to X$ be a continuous map. Given $x \in X$, we denote by $\Orbit(x)=\{T^{j}(x) : j \geq 0\}$ its positive orbit. For a periodic point $x \in X$, we denote by $\mu_{\Orbit(x)}$ the unique $T$-invariant probability measure supported in $\Orbit(x)$. These measures are called \textit{periodic}, and $\mathcal{M}_{T}^{\textrm{per}}$ denotes the set of periodic measures. 

Letting $F: X \to \R^{d}$ be a continuous potential, we use the following notation for Birkhoff sums:
$$
F^{(n)}:= F+F \circ T+...+F \circ T^{n-1}.
$$
Recall that the \emph{rotation set} of $F$ is defined as:
$$
R(F)=\left \{ \int F \dd \mu : \mu \in \mathcal{M}_{T} \right  \}.
$$
This is a compact convex subset of $\R^{d}$. Also, define the \emph{periodic rotation set} of $F$ as:
$$
R_{\per}(F)=\left \{ \int F \dd \mu : \mu \in \mathcal{M}_{T}^{\textrm{per}} \right \}.
$$
Clearly if $\mathcal{M}_{T}^{\textrm{per}}$ is dense in $\mathcal{M}_{T}$, then $R_{\per}(F)$ is dense in $R(F)$. Let us prove the continuity of the map $R$:
\begin{prop}\label{contfish}
The map $R:(C(X, \R^{d}), \norm{}_{\infty}) \rightarrow \mathit{CB}(\R^{d},d_{H})$ is continuous.
\end{prop} 
\begin{proof} Let $\mu \in \mathcal{M}_{T}$ and $F, G \in C(X, \R^{d})$, and note that:
\begin{equation*}
\begin{split}
\norm{\int F \dd \mu - \int G \dd \mu} \leq \norm{F-G}_{\infty}
\end{split}
\end{equation*}
and this immediately implies that $d_{H}(R(F), R(G))\leq \norm{F-G}_{\infty}$. \end{proof}

We end this section with a definition that we will use in the rest of the paper. Let $A \subset \R^{d}$ and $\varepsilon>0$. We define $\varepsilon$-neighbourhood of $A$ as $$B_{\varepsilon}(A)=\{x \in \R^{d} : \inf_{y \in A} \norm{x-y} < \varepsilon \}.$$

\section{Approximate Ma\~n\'e Lemma}

The \textit{Ma\~n\'e lemma} is a useful tool in ergodic optimization \cite{Sa, Bo1, CG, Je3,  Bo2}. It is stated as follows in the particular situation of expanding dynamics: let $T: X \to X$ be a expanding map and $\alpha \in (0,1]$. Then, for any $f$ in the space $C^{\alpha}(X)$ of $\alpha$-H\"{o}lder functions, there exists $h \in C^{\alpha}(X)$ such that 
$
\alpha(f) \leq f+h \circ T-h \leq \beta(f),
$
where $\alpha(f)$ and $\beta(f)$ are the minimum and maximum ergodic average, respectively.
This says that up to adding a coboundary $h-h\circ T$ to $f$ (which does not alter the integrals with respect invariant measures), we can assume that the image of $f$ is contained in the rotation set $R(f)=[\alpha(f), \beta(f)]$.

We can ask if there is an analogous of the Ma\~n\'e Lemma in the setting of vectorial potentials. Following the same spirit of the Ma\~n\'e Lemma, we say that a vectorial potential $F \in C(X,\R^{d})$ satisfies the \textit{Ma\~n\'e Lemma} if there exists $H \in C(X,\R^{d})$ such that $\Ima(F+H-H\circ T )\subset R(F)$. Even if we impose some regularity on $F$, the classical example of the fish is a H\"{o}lder function that does not satisfy the Ma\~n\'e Lemma, as noted by Bochi and Delecroix: see \cite[Proposition 2.1]{B}. 

Nevertheless, we have the following \textit{approximate Ma\~n\'e Lemma}:
\begin{lemma}\label{lemaaprox}
Let $F: X \rightarrow \R^{d}$ be a continuous function and $\varepsilon>0$. Then there exists a continuous function $G: X \to \R^{d}$ cohomologous to $F$ such that:
$$
\Ima \left ( G \right ) \subset B_{\varepsilon}(R(F)).
$$
Moreover, there exists $N_{0} \in \N$ such that $G$ can be taken to be $\bkav{F}$ for arbitrary $n \geq N_{0}$.
\end{lemma}
Lemma \ref{lemaaprox} is well known (c.f. ``enveloping property'' \cite[p. 6]{B}), but for completeness we give a proof. Note that if $F: X \rightarrow \R^{d}$ is continuous, then $F$ is cohomologous to $\bkav{F}$  for all $n \in \N$, since $F=\bkav{F}+H-H\circ T $, where $H= \frac{1}{n} \sum_{j=1}^{n} F^{(j)}$.
\begin{proof}[Proof of Lemma \ref{lemaaprox}]
Suppose in order to get a contradiction that there exists a number $\varepsilon>0$ and a sequence $\{x_{n} \}_{n \in \N}$ such that:
$$
\bkav{F}(x_{n}) \notin B_{\varepsilon}(R(F)) . 
$$
Consider the following sequence of probability measures on $X$:
$$
\mu_{n}=\frac{\delta_{x_{n}}+\delta_{T(x_{n})}+..+\delta_{T^{n-1}(x_{n})}}{n} .
$$
By compactness of the space of probability measures there exists a subsequence $\mu_{n_{k}}$ converging to a probability measure $\mu$. It is not hard to see that $\mu$ is a $T$-invariant probability measure. Thus, by the weak-$*$ convergence, we obtain:
$$
\int F \dd \mu_{n_{k}} \rightarrow \int F \dd \mu
$$
and since $B_{\varepsilon}(R(F))^{c}$ is closed, we have $\int F \dd \mu \notin  B_{\varepsilon}(R(F))$, a contradiction. Since $F$ is cohomologous to its finite time averages, we can take $G=\bkav{F}$ for sufficiently large $n$.
\end{proof}
\section{Proof of the main result}
In this section we present the main technical ingredients in the proof of Theorem \ref{thmC} and we combine them at the end. We will always assume that $(X,T)$ is a non-uniquely ergodic topological dynamical system with dense set of periodic measures. The first technical lemma enlarges the rotation sets, without losing the control of the distance to the original potential.
.
\begin{lemma}\label{constructlemma1}
Let $F: X \rightarrow \R^{d}$ be a continuous function with $\sharp R(F) \geq 2$ and $K \in \mathit{CB}_{\circ}(\R^{d}) $ such that $R(F) \subset \innt(K)$. Let $z_{1}, ..., z_{m}$ be distinct points in $R_{\per}(F)\setminus \partial R(F)$ and $y_{1}, ... , y_{m} \in \innt(K) \setminus R(F)$ be such that $R(F) \subset \conv\{y_{1}, ..., y_{n}\}$ (see Figure \ref{fig1}). Then there exists a continuous potential $G: X \rightarrow \R^{d} $ with:
\begin{enumerate}
    \item $\norm{G-F}_{\infty} \leq \frac{7}{6}\max_{i} \norm{z_{i}-y_{i}}$, and
    \vspace{0.1cm}
    
    \item $ \conv \{ y_{1}, ... ,y_{m} \} \subset R(G) \subset \innt(K)$.
\end{enumerate}
\begin{figure}[ht]
\begin{tikzpicture}[scale=1]
\draw [rotate around={0:(1.37,0)},line width=1pt] (1.37,0) ellipse (2.9137626021290584cm and 1.9872971341590677cm);
\usetikzlibrary{patterns}
\draw [rotate around={0:(1.4144063448758233,0)},line width=0.8pt,fill=black,fill opacity=0.03] (1.4144063448758233,0) ellipse (2.0546028875538758cm and 0.9469147392639833cm);;
\draw (1.4318865584417144,0.1) node{\footnotesize{$R(F)$}};
\draw (1.8157096014692054,-1.5325348468170308) node{\small{$K$}};
\draw (0.053008719,0.4365) node[below] {\small{$z_{1}$}};
\draw (0.053008719,0.4365) node {$\bullet$};

\draw (0.153008719,1.4365) node[below] {\small{$y_{1}$}};
\draw (0.053008719,1.4365) node {$\bullet$};

\draw (1.4318865584417144,0.766272187101962) node[below] {\small{$z_{2}$}};
\draw (1.4318865584417144,0.7866272187101962) node {$\bullet$};

\draw (1.4318865584417144,1.566272187101962) node[below] {\small{$y_{2}$}};
\draw (1.4318865584417144,1.566272187101962) node {$\bullet$};

\draw (2.8318865584417144,0.3365) node[below] {\small{$z_{3}$}};
\draw (2.8318865584417144,0.3365) node {$\bullet$};

\draw (2.9318865584417144,1.1365) node[below] {\small{$y_{3}$}};
\draw (3.0318865584417144,1.1365) node {$\bullet$};

\draw (3.0318865584417144,-0.1365) node[below] {\small{$z_{4}$}};
\draw (3.0318865584417144,-0.1365) node {$\bullet$};

\draw (3.8318865584417144,-0.1365) node[below] {\small{$y_{4}$}};
\draw (3.8318865584417144,-0.1365) node {$\bullet$};

\draw (2.3318865584417144,-0.4365) node[below] {\small{$z_{5}$}};
\draw (2.3318865584417144,-0.4365) node {$\bullet$};

\draw (2.3318865584417144,-1.2365) node[below] {\small{$y_{5}$}};
\draw (2.3318865584417144,-1.2365) node {$\bullet$};

\draw (1.053008719,-0.5365) node[below] {\small{$z_{6}$}};
\draw (1.053008719,-0.5365) node {$\bullet$};

\draw (0.253008719,-1.1365) node[below] {\small{$y_{6}$}};
\draw (0.253008719,-1.1365) node {$\bullet$};

\draw (0.053008719,-0.2365) node[below] {\small{$z_{7}$}};
\draw (0.053008719,-0.2365)node {$\bullet$};

\draw (-0.993008719,-0.0365) node[below] {\small{$y_{7}$}};
\draw (-0.993008719,-0.0365) node {$\bullet$};

\draw[densely dotted] (0.053008719,1.4365) -- (1.4318865584417144,1.566272187101962);

\draw[densely dotted] (1.4318865584417144,1.566272187101962) -- (3.0318865584417144,1.1365);

\draw[densely dotted] (3.0318865584417144,1.1365) -- (3.8318865584417144,-0.1365);

\draw[densely dotted] (3.8318865584417144,-0.1365) -- (2.3318865584417144,-1.2365);

\draw[densely dotted] (2.3318865584417144,-1.2365) -- (0.253008719,-1.1365);

\draw[densely dotted] (0.253008719,-1.1365)  -- (-0.993008719,-0.0365);

\draw[densely dotted] (-0.993008719,-0.0365) -- (0.053008719,1.4365);





\end{tikzpicture}
\caption{Setting for Lemma \ref{constructlemma1} with $m=7$.}\label{fig1}
\end{figure}
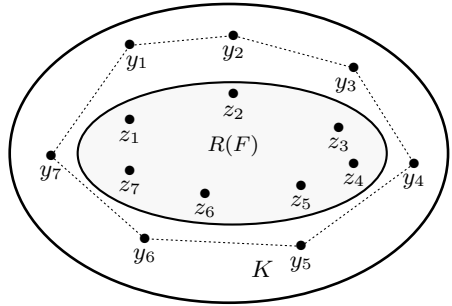
\begin{proof}
Fix $\varepsilon>0$ such that $B_{\varepsilon}(R(F)) \subset \innt(K)$ and $y_{j} \notin B_{\varepsilon}(R(F))$ for all $j=1, ..., m$. Thus, we can apply Lemma \ref{lemaaprox} to $F$ in order to obtain $n \in \N$ such that:
\begin{itemize}
    \item $ \Ima \left ( \bkav{F} \right ) \subset B_{\varepsilon}(R(F)) \subset \innt(K)$
 \item \label{terceraprop}  For each $z \in \{z_{1}, ..., z_{m}\}$ there exists a periodic point $x \in X$ such that the Birkhoff average $\bkav{F} $ equals $z$ on the orbit of $x$. We denote by $x_{j}$ the corresponding point to $z_{j}$. 
\end{itemize}
The main idea is to perturb the potential $F$ nearby the periodic orbits. For this purpose, let us choose the index set $I=\{(i,j) : 1 \leq i \leq m , 1 \leq j \leq \sharp \Orbit(x_{i}) \} $ and a collection of open balls $\{B_{(i,j)}\}_{(i,j) \in I} \subset X$ centered at the periodic points defined by:
$$
B_{(i,j)}=B_{r}(T^{j}(x_{i})) \hspace{0.3cm} \forall (i,j) \in I
$$
 and $r>0$  sufficiently small so that the collection of balls $\{B_{(i,j)}\}_{(i,j) \in I}$ are pairwise disjoint, $\bkav{F}(B_{i,j}) \subset \innt(R(F))$ and:
\begin{equation}\label{diameq}
\textrm{diam}\left (\conv\left \{\bkav{F}(B_{i,j}) \cup \{y_{i} \} \right \}\right ) \leq \frac{7}{6} \norm{y_{i}-z_{i}} 
\end{equation}
Let $B_{*}$ be the complement of $\Orbit(x_{1}) \cup ... \cup \Orbit(x_{m})$. Take a continuous partition of unity
$$
\rho_{*}+\sum_{(i,j) \in I} \rho_{i,j}=1
$$
subordinated to the open cover $B_{*} \cup \bigcup_{(i,j) \in I}B_{(i,j)}=X$.
Next, we define a function $\widetilde{G}: X \to \R^{d}$ as:
$$
\widetilde{G}(x)= \sum_{(i,j) \in I} \rho_{i.j}(x)  y_{i} + \rho_{*}(x) \bkav{F}(x).$$
We claim that $\widetilde{G}$ satisfies very similar properties as in the statement of the lemma. First, note that $\widetilde{G}$ is constant equal to $y_{i}$ on $\Orbit(x_{i})$. which implies $y_{i} \in R(\widetilde{G})$ for every $i=1, ..., m$. Therefore, $\conv\{y_{1}, ..., y_{m}\} \subset R(\widetilde{G})$. Now,
\begin{equation*}
\forall x \in X, \hspace{0.2cm} \widetilde{G}(x) \in \conv\left \{ \{y_{1}, ..., y_{m} \} \cup \Ima\left ( \bkav{F} \right )\right \} ,
\end{equation*}
since $\widetilde{G}$ is a convex combination of $y_{1}, ... ,y_{m}$ and $\bkav{F}$. The later implies:
$$
R(\widetilde{G}) \subset \conv\left \{ \{y_{1}, ..., y_{m} \} \cup \Ima\left ( \bkav{F} \right )\right \} \subset \innt{K} .
$$
Consequently, $R(F) \subset \conv \{ y_{1}, ... ,y_{m} \} \subset R(\widetilde{G}) \subset \innt(K)$. The next step is to estimate the distance between $\widetilde{G}$ and $\bkav{F}$. Let $x \in X$:
\begin{itemize}
\item If $x \in B_{i,j}$ then $\widetilde{G}(x)=\rho_{i,j}(x)y_{i}+(1-\rho_{i,j}(x))\bkav{F}(x)$, and therefore, using \eqref{diameq} $$\norm{\widetilde{G}(x)-\bkav{F}(x)}=|\rho_{i,j}(x)|\norm{y_{i}-\bkav{F}(x)}\leq \frac{7}{6} \norm{y_{i}-z_{i}} .$$ 
\item If $x \notin \bigcup_{(i,j) \in I}B_{i,j}$, then $\widetilde{G}(x)=\bkav{F}(x)$.
\end{itemize}
We conclude that $\norm{\widetilde{G}-\bkav{F}}_{\infty} \leq \frac{7}{6} \max_{i}\norm{z_{i}-y_{i}}$. Now consider
$$
G=\widetilde{G}+\left (F-\bkav{F} \right) .
$$
Recall that $F-\bkav{F}$ is a coboundary. Therefore, $G$ has the same rotation set of $\widetilde{G}$, which is sandwiched between $\conv \{y_{1}, ... ,y_{m} \}$ and $\innt(K)$. Furthermore,
\begin{equation*}
\norm{G-F}_{\infty}=\norm{\widetilde{G}-\bkav{F}}_{\infty} \leq \frac{7}{6} \max_{i}\norm{z_{i}-y_{i}} . \qedhere
\end{equation*}
\end{proof}
\end{lemma}
At this moment, we have a technical tool to enlarge rotation sets and control the distance between the potentials. Now we will upgrade the previous lemma also considering the distance between the convex bodies.
\begin{lemma}\label{constructlemma2}
Let $F: X \rightarrow \R^{d}$ be a continuous function, let $K \in \mathit{CB}_{\circ}(\R^{d}) $ be such that $R(F) \subset \innt(K)$, and let $\varepsilon=d_{H}(R(F), K)$. Then there exists a continuous function $G: X \rightarrow \R^{d}$  with the following properties:
\begin{enumerate}
\item \label{constructlemma21} $R(F) \subset R(G) \subset \innt(K)$
\item \label{constructlemma22}$ \norm{G-F}_{\infty} \leq \kappa \varepsilon$
\item \label{constructlemma23} $d_{H}(R(G), K) \leq \kappa \varepsilon$
\end{enumerate}
where $\kappa=\frac{29}{30}$.
\begin{proof}
The main idea is to take a polytope which is sufficiently close to $R(F)$ and then apply Lemma \ref{constructlemma1}. But this is not sufficient to ensure that condition  (\ref{constructlemma23}) is satisfied, so we need to enlarge the polytope in order to have the new rotation set moderately close to $K$.

First suppose that $R(F)$ is not a singleton. Fix $\delta \in (0,\frac{\varepsilon}{5})$ with $\overline{B_{\delta}(R(F))} \subset \innt{K}$. Now, by \cite[Theorem 1.8.16]{S} we can take distinct points $y_{1}, ..., y_{\ell} \in B_{\delta}(R(F)) \setminus R(F)$ such that $R(F) \subset \conv \{y_{1}, ..., y_{\ell}\}$. Due to the compactness of $\partial K$, we may choose distinct points $w_{\ell+1}, ..., w_{m} \in \partial K$ such that:
$$
\partial K \subset \bigcup_{j=\ell+1}^{m} B_{\frac{\varepsilon}{4}}(w_{j}).
$$
Hence, since $d_{H}(R(F),K)=\epsilon$, there exist distinct points  $y_{\ell+1}, ..., y_{m} \in \innt(K) \setminus R(F)$ with $\norm{y_{j}-w_{j}}\leq \frac{\varepsilon}{3}$ and $d(R(F),y_{j}) \leq \frac{2\varepsilon}{3}$ for each $j=\ell+1,...,m$. Since $R_{\per}(F)$ is dense in $R(F)$ which by assumption is not a singleton, we can also find distinct points $z_{1}, ..., z_{m} \in R_{\per}(F) \setminus \partial R(F)$ such that:
$$
\norm{z_{j}-y_{j}}\leq \frac{4\varepsilon}{5}
$$
for all $j=1,..., m$. By Lemma \ref{constructlemma1}, we can perturb $F$, and obtain a continuous $G: X \rightarrow \R^{d}$ such that:

 $$\norm{G-F}_{\infty}\leq \frac{7}{6} \cdot\frac{4\varepsilon}{5}=\frac{28\varepsilon}{30}$$
 and
$$R(F) \subset \conv( \{y_{1}, ..., y_{m} \} )\subset R(G) \subset \innt(K) .$$

So conditions (\ref{constructlemma21}) and (\ref{constructlemma22}) are satisfied. In order to check the reamining condition (\ref{constructlemma23}), note first that $d_{H}(K, R(G))=d_{H}(\partial K, \partial R(G))$. Let $x \in \partial K$. Then there exists $w_{j} \in \partial K$ such that $w \in B_{\frac{\varepsilon}{4}}(w_{j})$. So:
\begin{equation*} \label{eq3}
\begin{split}
d(w,\partial R(G)) &\leq \norm{w-y_{j}} \\ &\leq \norm{w-w_{j}} + \norm{w_{j}-y_{j}}\\
&\leq \frac{\varepsilon}{4}+\frac{\varepsilon}{3} \\
&\leq \frac{28 \varepsilon}{30}.
\end{split}
\end{equation*}
Therefore $d_{H}(\partial K, \partial R(G))\leq \frac{28}{30}\varepsilon$ and this implies condition (\ref{constructlemma23}).\\ 
For the case when $R(F)$ is a singleton,  consider a continuous perturbation $F{'}$ of $F$ near two disjoint periodic orbits, say $\Orbit(x_{1})$ and $\Orbit(x_{2})$, such that:
\begin{itemize}
\item $\int F{'} \dd \mu_{\Orbit(x_{1})} \neq \int F{'} \dd \mu_{\Orbit(x_{2})}$
\item $R(F{'}) \subset \innt(K)$
\item $\norm{F-F{'}}_{\infty} \leq 0.01 \varepsilon$
\item $d_{H}(R(F{'}),K) \leq 1.01\varepsilon$ 
\end{itemize}
and apply the same procedure as before to $F^{'}$.
\end{proof}
\end{lemma}
As a straightforward consequence of the previous lemma, we have the following proposition:
\begin{prop}\label{propbacan}
Let $F: X \rightarrow \R^{d}$ be a continuous function and let $K \in \mathit{CB}(\R^{d}) $ such that $R(F) \subset \relinnt(K)$. Then, there exists a continuous function $G: X \rightarrow \R^{d}$ such that $\norm{F-G}_{\infty} \leq Cd_{H}(R(F), K)$ and $R(G)=K$, where $C=30$.
\end{prop}
\begin{proof}
We divide the proof in two cases. First suppose that $\innt(K) \neq \emptyset$. Apply Lemma \ref{constructlemma2} recursively  to obtain a sequence of continuous functions $F_{n}: X \rightarrow \R^{d}$ such that:
\begin{itemize}
\item $d_{H}(F_{n},K)\leq \kappa^{n}d_{H}(R(F),K) $
\item $\norm{F_{n}-F_{n+1}}_{\infty} \leq \kappa^{n}d_{H}(R(F),K)$
\end{itemize}
where $F_{1}=F$. Then $\{ F_{n} \}_{n \in \N}$ is a Cauchy sequence, and therefore converges to a continuous function $G: X \rightarrow \R^{d}$ which satisfies:
$$
\norm{F-G}_{\infty} \leq \sum_{j=1}^{\infty} \norm{F_{j+1}-F_{j}}_{\infty} \leq \sum_{j=1}^{\infty}\kappa^{j}d_{H}(R(F),K) \leq 30 d_{H}(R(F),K).
$$ By Proposition \ref{contfish}, the map $R: C(X, \R^{d}) \rightarrow \mathit{CB}(\R^{d})$ is continuous, and so:
$$
R(G)=R(\lim F_{n})=\lim R(F_{n})=K. 
$$
Thus, the proof of the first case is finished. Now suppose that $\innt(K)=\emptyset$. Let $\mathcal{P}(K)$ be the least affine hyperspace passing through $K$. We can consider $F$ as a function taking values in $\mathcal{P}(K)$ and this affine hyperplane can be identified with $\R^{\ell}$, where $\ell=\dim \mathcal{P}(K)$. In this situation we can see $K$ as a subset of this $\R^{\ell}$ with $\innt(K) \neq \emptyset$.  Consequently, the proof is reduced to the first case. \qedhere
\end{proof}
Now we need an adjustment in order to drop the hypothesis $R(F) \subset \relinnt{K}$.
\begin{lemma}\label{adjustmentlemma}
Let $F: X \rightarrow \R^{d}$ be a continuous function, $K \in \mathit{CB}(\R^{d})$, and $\varepsilon>0$. Suppose that $d_{H}(R(F), K) \leq \varepsilon$. Then there exists a continuous function $F{}': X \rightarrow \R^{d}$ with:
\begin{itemize}

\item[1)] $R(F{}') \subset \relinnt(K)$ 
\item[2)] $\norm{F-F{}'}_{\infty} \leq 2 \varepsilon$
\item[3)] There exists a continuous function $F{}'': X \to \R^{d}$ cohomologous to $F{}'$ such that $\Ima(F{}'') \subset \mathcal{P}(K)$, where $\mathcal{P}(K)$ is the least affine hyperspace containing K.
\end{itemize}
\end{lemma}
\begin{proof}
The strategy is similar of the proof Lemma \ref{constructlemma1}. Apply Lemma \ref{lemaaprox} to $F$ and $\varepsilon>0$ to obtain $n \in \N$ with $\Ima \left (\frac{F^{(n)} }{n} \right ) \in B_{\varepsilon}(R(F)) $. Also, apply Lemma \ref{preadjustlemma} to $K$ and $\delta=\min \{ \varepsilon, \frac{1}{2} \}$ to find $L \in \mathit{CB}(\R^{d})$ with $L \subset \relinnt(K)$ and $d_{H}(K,L)\leq \delta$. Define $F{}''$ as:
$$
F{}''= P_{L}\left (\bkav{F} \right ),
$$
where $P_{L}$ is the projection onto $L$, that is, the map which sends each point of the space to its closest point in $L$. Since $P_{L}$ is Lipschitz, the function $F{}''$ is continuous. Also $R(F{}'') \subset \relinnt(K)$, so the next step is to estimate $d_{H}(R(F{}''),K)$. Given $y \in K$, due to the denseness of $R_{\per}(\bkav{F})$ in $R(F)$,  there exists $z \in R_{\per}(\bkav{F})$ such that $\norm{y-z} \leq 2\varepsilon $. Let $\Orbit(x)$ be the corresponding periodic orbit. We note that:

\begin{equation*}
\begin{split}
\norm{y-\int F{}'' \dd \mu_{\Orbit(x)}} &\leq \norm{y-z}+\norm{z-\int F{}'' \dd \mu_{\Orbit(x)}} \\
&\leq 2\varepsilon + \norm{\int\bkav{F}-F{}'' \dd \mu_{\Orbit(x)}} \\
& \leq 2\varepsilon + \int 2\varepsilon \dd \mu_{\Orbit(x)}\\
 &\leq 4\varepsilon ,
\end{split}
\end{equation*}
since  $\norm{\bkav{F}-F{}''}_{\infty}\leq 2\varepsilon$. From above we get that $d_{H}(K, R(F{}'')) \leq 4\varepsilon$. Now, it suffices to consider $F{}'= F{}''+(F-\bkav{F})$, which is cohomologous to $F{}''$. Finally,
\begin{equation*}
\norm{F-F{}'}_{\infty}=\norm{F{}''-\bkav{F}}_{\infty}\leq 2\varepsilon. \qedhere
\end{equation*}
\end{proof}

Now we are ready to prove Theorem \ref{thmC}.
\begin{proof}[Proof of Theorem \ref{thmC}]
Let $F \in C(X, \R^{d})$ and $\varepsilon >0$. Let $K \in \mathit{CB}(\R^{d})$ such that $d_{H}(K, R(F)) \leq \varepsilon$. Let $F{}'$ and $F{}''$ given by Lemma \ref{adjustmentlemma}. Then apply Proposition \ref{propbacan} to $F{}''$ in order to obtain a continuous function $\widetilde{G}: X \rightarrow \R^{d}$ with the properties that $R(\widetilde{G})=K$ and $\norm{F{}''-\widetilde{G}}_{\infty} \leq 4 C\varepsilon$. So, we define:
$$
G=\widetilde{G}+(F{}'-F{}'') .
$$
Hence $R(G)=K$, since  $F{}'$ is cohomologous to $F{}''$. Moreover,
$$\norm{G-F{}'}_{\infty}=\norm{\widetilde{G}-F{}''}_{\infty} \leq 4C\varepsilon .$$
Therefore:
$$
\norm{F-G}_{\infty} \leq \norm{F-F{}'}_{\infty}+\norm{F{}'-G}_{\infty} \leq 2\varepsilon+4C\varepsilon=(2+4C)\varepsilon.
$$
We have just proved that
$$
R(\overline{B}_{(2+4C)\varepsilon}(F)) \supset \overline{B}_{\varepsilon}(R(F)) ,
$$
and this inclusion implies the openness of $R$.
The surjectivity follows directly from Proposition \ref{propbacan}. Let $K \in \mathit{CB}(\R^{d})$, $v \in \relinnt(K)$ and $F \equiv v$. Thus, applying Proposition \ref{propbacan} to $F$, we get a continuous function $G \in C(X, \R^{d})$ such that $R(G)=K$.
\end{proof}
\section{Directions for further research}
In this section we discuss related problems and open questions.
\subsection{The uniqueness property} Let $F \in C(X,\R^{d})$. We say that $F$ satisfies the \textit{uniqueness property} if for each $v \in \partial R(F)$, there exists a unique $\mu \in \mathcal{M}_{T}$ for which $\int F \dd \mu = v$. As mentioned in the introduction, in the one-dimensional case, generic functions $f \in C(X,\R)$ satisfy the uniqueness property. So we ask:
\begin{qstn}\label{q1}
Is it true that generic functions $F \in C(X,\R^{d})$ satisfy the uniqueness property?
\end{qstn}
Of course, we can replace $C(X,\R^{d})$ for other spaces of functions. Following the proof in the one-dimensional case in \cite[Theorem 3.2]{Je4}, one can show the following:
\begin{prop}
The set of $F \in C(X,\R^{d})$ which satisfy the uniqueness property is a $G_{\delta}$ set.
\end{prop}
Therefore in order to give a positive answer to  Question $\ref{q1}$, it is sufficient to prove denseness.
\subsection{The map $R(\cdot)$ is not open in general} \label{sec6.2} It is natural to ask if the map $R(\cdot)$ is open if we replace $C(X,\R^{d})$ by other spaces of functions. The answer is negative in the space of Lipschitz functions: Let $\mathrm{Lip}(X, \R^{2})$ be the subspace of Lipschitz potentials endowed with the Lipschitz norm
$$
\norm{f}_{\Lip}=\norm{f}_{\infty}+\sup_{x \neq y} \frac{|f(x)-f(y)|}{d(x,y)} .
$$
Then, we have the following:
\begin{prop}\label{propnoabierto}
Suppose that $T$ has a fixed point $x_{0}$. Then there exists an open set $U \subset \mathrm{Lip}(X, \R^{d})$ such that for all $F \in U$, $\partial R(F)$ is non-differentiable. In particular, the restriction $R|_{\mathrm{Lip}(X, \R^{2})}$ is not open.
\end{prop}
\begin{proof}
The proof follows the same spirit as \cite[Proposition 4.12]{Je1}. Let $F(x)=(0,-2d(x,x_{0}))$ and $U=B_{\frac{1}{2}}(F)$. Let $G \in \mathrm{Lip}(X, \R^{2})$ be a Lipschitz perturbation of $F$ with $\norm{G}_{\mathrm{Lip}}<\frac{1}{2}$. We claim that $(F+G)(x_{0})$ is a corner of $R(F+G)$. Since the rotation map is equivariant with respect to translations, we can assume that $G(x_{0})=(0,0)$. Thus,
$$
(1,1) \cdot (F+G)(x) \leq -2d(x,x_{0})+\sqrt{2}G(x) \leq -2d(x,x_{0})+\frac{\sqrt{2}}{2}d(x,x_{0}) \leq 0
$$
Analogously $(1,-1) \cdot (F+G)(x) \leq 0$. We conclude that $\delta_{x_{0}}$ is a maximizing measure for $(1,\pm 1) \cdot (F+G)$, thus:
$$
\int (F+G) \dd \delta_{0}=(0,0)
$$
is a corner for $R(F+G)$, because $R(F+G)$ contains $(0,0)$ and is contained in the cone $\{(x,y) \in \R^{2} : y \leq - |x| \}$ with vertex $(0,0)$. Since convex bodies with $C^{1}$ boundary is dense, we conclude that $R|_{\mathrm{Lip}(X, \R^{2})}$ is not open at $F$.
\end{proof}
From this proposition, we also conclude that differentiability of the rotation set boundary is not generic when we consider the space of Lipschitz functions.

\subsection{Genericity result for other spaces}
In this article we considered the case of continuous potentials. We propose to investigate the same question for other spaces of functions and other dynamics:
\begin{qstn}
Is it true that the rotation set is strictly convex for generic potentials in some dense subspace of $C(X, \R^{d})$ ?
\end{qstn}
For example, replace $C(X,\R^{d})$ by the space of $\alpha$-H\"{o}lder potentials $C^{\alpha}(X,\R^{d})$ with the H\"{o}lder norm. Also, in view of the fish example and Proposition \ref{propnoabierto}, it seems that if we impose regularity to the potential, then the corresponding rotation set $R(F)$ is going to have a considerable number of corners in the boundary. So, we pose the following:
\begin{qstn}
It is true that the boundary of rotation set has a (full measure) dense subset of corners for generic potentials in $C^{\alpha}(X, \R^{d})$ ?
\end{qstn}
For more discussion, see \cite[Section 2]{B}.

\hspace{-4.3mm}\textbf{Acknowledgment}\hspace{2mm} I am grateful to my advisor J. Bochi for very valuable discussions and corrections throughout all this work. I would also like to thank the referee for their useful corrections and suggestions. This article was supported by CONICYT scholarship 22181136 and CONICYT PIA ACT172001.

\small{Sebasti\'an Pavez-Molina (\texttt{snpavez@uc.cl})}\\
\small{Facultad de Matem\'aticas}\\
\small{Pontificia Universidad Cat\'olica de Chile}\\
\small{Av. Vicu\~na Mackenna 4860 Santiago Chile}


\begin{thebibliography}{KLP2}

\bibitem[B]{B}
\textsc{Bochi, J.} --
Ergodic optimization of Birkhoff averages and Lyapunov exponents. \textit{Proc. Int. Cong. of Math. - 2018 Rio de Janeiro} Vol. 2, 1821--1842

\MRbibitem[BCH]{3502068}{BCH}
\textsc{Boyland, P.; Carvalho, A; Hall T.} --
New rotation sets in a family of torus homeomorphisms \textit{Inventiones Mathematicae} 204 (2016), 895--937.


\MRbibitem[Bo1]{1785392}{Bo1}
\textsc{Bousch, T.} --
Le poisson n'a pas d'ar\^{e}tes. \textit{Ann. Inst. H. Poincar\'e Probab. Statist.} 36 (2000), 489--508

\MRbibitem[Bo2]{1936826}{Bo2}
\textsc{Bousch, T.} --
Um lemme de Ma\~n\'e bilat\'eral. \textit{CRAS 335} (2002), 533--536.


\bibitem[CG]{CG}
\textsc{Conze, J.P.; Guirvac'h, Y.} --
Croissance des sommes ergodiques et principe variationnel. Unpublished manuscript, circa 1993.


\MRbibitem[GK]{3820000}{GK}
\textsc{Gelfert, K.; Kwietniak, D.}, 
On density of ergodic measures and generic points. 
\textit{Ergodic Theory Dynam. Systems} 
38 (2018), no. 5, 1745--1767



\MRbibitem[Je1]{1780215}{Je4}
\textsc{Jenkinson, O. },
\textit{Conjugacy rigidity, cohomological triviality, and barycentres
of invariant measures.} PhD thesis, Univ. of Warwick, 1996.

\MRbibitem[Je2]{2191393}{Je1} \textsc{Jenkinson, O.} --Ergodic Optimization. \textit{Discr. Cont. Dyn. Syst.} 15 (2006), 197--224

\MRbibitem[Je3]{2226487}{Je2}\textsc{Jenkinson, O.} --Every ergodic measure is uniquely maximizing. \textit{Discr. Cont. Dyn. Syst.} 16 (2006), 383--392

\MRbibitem[Je4]{4000508}{Je3}
\textsc{Jenkinson, O.} --
Ergodic optimization in dynamical systems.
\textit{Ergodic Theory Dynam. Systems} 39 (2019), no. 10, 2593--2618



\MRbibitem[KW]{3219558}{KW}
\textsc{Kucherenko, T., Wolf}, C. Geometry and entropy of generalized rotation sets. 
\textit{Israel J. Math.}
199 (2014), 791 -- 829

\MRbibitem[KliN]{631603}{KlN}
\textsc{Klima, V. ;  Netuka, I.} --
Smoothness of a typical convex function.
\textit{Czecboslovak Math. J.} 31(106), 569--572.

\MRbibitem[MK]{1053617}{MK}
\textsc{Misiurewicz, M.; Ziemian K.} --
Rotation sets for maps of tori.
\textit{J. London Math. Soc. (2)} 40 (1989), no. 3, 490--506

\MRbibitem[P]{3174742}{P}
\textsc{Passegi, A.} --
Rational polygons as rotation sets of generic homeomorphisms of the two torus.
\textit{J. London Math. Soc.} (1) 89 (2014), 235 -- 254.

\MRbibitem[Sa]{1724277}{Sa}
\textsc{Savchenko, S. V. } -- Homological inequalities for finite topological Markov chains. \textit{Funct. Anal. Appl.} 33 (1999), no. 3, 236--238

\bibitem[S]{S}
\textsc{Schneider, R.} \textit{Convex bodies: the Brunn-Minkowski theory.} 
Cambridge university press, 2014.






\MRbibitem[Zi]{1314983}{Zi}
\textsc{Ziemian, K.}
Rotation sets for subshifts of finite type. \textit{Fund. Math.} 146 (1995), no. 2, 189--201


\end{thebibliography}
\end{document}